\documentclass[12pt]{amsart}

\usepackage{amssymb}
\usepackage{bm}
\usepackage{graphicx}
\usepackage{psfrag}
\usepackage{hyperref}
\hypersetup{colorlinks=true, linkcolor=blue, citecolor=red, urlcolor=wine}
\usepackage{color}
\usepackage{url}
\usepackage{algpseudocode}
\usepackage{fancyhdr}
\usepackage{xy}
\input xy
\xyoption{all}
\usepackage{stmaryrd}

\newtheorem*{maintheorem*}{Main Theorem}
\newtheorem{theorem}{Theorem}[section]
\newtheorem{prop}[theorem]{Proposition}
\newtheorem{lemma}[theorem]{Lemma}
\newtheorem{cor}[theorem]{Corollary}

\newtheorem{question}[theorem]{Question}

\theoremstyle{definition}

\newtheorem{example}[theorem]{Example}

\numberwithin{equation}{section}

\newcommand{\nn}{\mathbb{N}}
\newcommand{\qq}{\mathbb{Q}}
\newcommand{\rr}{\mathbb{R}}
\newcommand{\zz}{\mathbb{Z}}
\providecommand\ldb{\llbracket}
\providecommand\rdb{\rrbracket}

\newcommand{\gp}{\mathsf{gp}}

\hypersetup{colorlinks=true, linkcolor=blue, citecolor=red, urlcolor=wine}



\begin{document}
	
	\mbox{}
	\title{On the atomicity of monoid algebras}
	
	\author{Jim Coykendall}
	\address{Department of Mathematical Sciences\\Clemson University\\Clemson, SC 29634}
	\email{jcoyken@clemson.edu}
	
	\author{Felix Gotti}
	\address{Department of Mathematics\\UC Berkeley\\Berkeley, CA 94720 \newline \indent Department of Mathematics\\Harvard University\\Cambridge, MA 02138}
	\email{felixgotti@berkeley.edu}
	\email{felixgotti@harvard.edu}
	
	\date{\today}
	
\begin{abstract}
	Let $M$ be a commutative cancellative monoid, and let $R$ be an integral domain. The question of whether the monoid ring $R[x;M]$ is atomic provided that both $M$ and $R$ are atomic dates back to the 1980s. In 1993, Roitman gave a negative answer to the question for $M = \nn_0$: he constructed an atomic integral domain $R$ such that the polynomial ring $R[x]$ is not atomic. However, the question of whether a monoid algebra $F[x;M]$ over a field $F$ is atomic provided that $M$ is atomic has been open since then. Here we offer a negative answer to this question. First, we find for any infinite cardinal $\kappa$ a torsion-free atomic monoid $M$ of rank $\kappa$ satisfying that the monoid domain $R[x;M]$ is not atomic for any integral domain $R$. Then for every $n \ge 2$ and for each field $F$ of finite characteristic we exhibit a torsion-free atomic monoid of rank $n$ such that $F[x;M]$ is not atomic. Finally, we construct a torsion-free atomic monoid $M$ of rank $1$ such that $\zz_2[x;M]$ is not atomic.
\end{abstract}

\maketitle

\medskip
\section{Introduction}

 In~\cite{rG84} R.~Gilmer offers a very comprehensive summary of the theory of commutative semigroup rings developed until the 1980s. Many algebraic properties of a commutative ring $R$ and a monoid $M$ implying the corresponding property on the monoid ring $R[x;M]$ had been studied by that time, but still the following fundamental question was stated by Gilmer as an open problem.

\begin{question} \cite[p.~189]{rG84} \label{prob:Gilmer transfer of atomicity}
	Let $M$ be a commutative cancellative monoid and let $R$ be an integral domain. Is the monoid domain $R[x;M]$ atomic provided that both $M$ and~$R$ are atomic?
\end{question}

In 1990, D.~Anderson et al. restated a special version of the above question in the context of polynomial rings, namely the case of $M = (\nn_0,+)$.

\begin{question} \cite[Question~1]{AAZ90} \label{conj:atomicity transfer for polynomial rings}
	If $R$ is an atomic integral domain, is the integral domain~$R[x]$ also atomic?
\end{question}

Question~\ref{conj:atomicity transfer for polynomial rings} was answered negatively by M.~Roitman in 1993. He constructed a class of atomic integral domains whose polynomial rings fail to be atomic~\cite{mR93}. In a similar direction, Roitman constructed examples of atomic integral domains whose corresponding power series rings fail to be atomic as well as examples of atomic power series rings over non-atomic integral domains~\cite{mR00}. Observe that Roitman's negative answer to Question~\ref{conj:atomicity transfer for polynomial rings} gives a striking answer to Question~\ref{prob:Gilmer transfer of atomicity}, showing that $R[x;M]$ can fail to be atomic even if one takes $M$ to be the simplest nontrivial atomic monoid, namely $M = (\nn_0, +)$. This naturally suggests the question of whether the atomicity of $M$ implies the atomicity of $R[x;M]$ provided that $R$ is the taken to be one of the simplest nontrivial atomic integral domains, a field. Clearly, this is another special version of Question~\ref{prob:Gilmer transfer of atomicity}.

\begin{question} \label{quest:main question}
	Let $F$ be a field. If $M$ is an atomic monoid, is the monoid algebra~$F[x;M]$ also atomic?
\end{question}

There are many known classes of atomic monoids whose monoid algebras (over any field) happen to be atomic. For instance, if a monoid $M$ satisfies the~ACCP (and, therefore, is atomic), then it is not hard to argue that for any field $F$ the monoid algebra $F[x;M]$ also satisfies the~ACCP. Therefore $F[x;M]$ inherits the atomicity of~$M$ as it is well known that integral domains satisfying the ACCP are atomic. In particular, every finitely generated monoid satisfies the ACCP and, thus, induces atomic monoid algebras. As BF-monoids also satisfy the ACCP~\cite[Corollary~1.3.3]{GH06}, one can use them to obtain many non-finitely generated atomic monoid algebras with rational (or even real) exponents; this is because submonoids of $(\rr_{\ge 0},+)$ not having $0$ as a limit point are BF-monoids~\cite[Proposition~4.5]{fG19}. Furthermore, an infinite class of atomic submonoids of $(\qq_{\ge 0},+)$ (which are not BF-monoids) with atomic monoid algebras was exhibited in~\cite[Theorem~5.4]{fG18}.

The purpose of this paper is to provide a negative answer for Question~\ref{quest:main question}. In Section~\ref{sec:non-atomic monoid domains} we find, for any infinite cardinal $\kappa$, a torsion-free atomic monoid $M$ of rank $\kappa$ satisfying that $R[x;M]$ is not atomic for any integral domain~$R$. Then, in Section~\ref{sec:non-atomic algebras of finite characteristic}, for every $n \ge 2$ and for each field $F$ of finite characteristic we exhibit a torsion-free atomic monoid of rank $n$ such that $F[x;M]$ is not atomic. Finally, in Section~\ref{sec:PM}, we construct a torsion-free atomic monoid $M$ of rank $1$ such that $\zz_2[x;M]$ is not atomic.
\medskip

\section{Background}
\label{sec:background}

\subsection{General Notation} Throughout this paper, we set $\mathbb{N} := \{1,2, \dots\}$ and we set $\nn_0 := \nn \cup \{0\}$. In addition, the symbols $\zz$, $\qq$, and $\rr$ denote the sets of integers, rational numbers, and real numbers, respectively. For $S \subseteq \rr$ and $r \in \rr$, we set
\[
	S_{\ge r} := \{s \in S \mid s \ge r\}
\]
and, in a similar manner, we use the notation $S_{> r}$. For $x,y \in \rr$ such that $x \le y$, we let $\ldb x,y \rdb$ denote the discrete interval between $x$ and $y$, i.e.,
\[
	\ldb x,y \rdb := \{ n \in \zz \colon x \le n \le y \}.
\]
Information background and undefined terms related to atomic monoids and domains can be found in~\cite{rG84}.

\subsection{Atomic Monoids.} Every monoid here is tacitly assumed to be commutative and cancellative. Unless we specify otherwise, monoids are written additively. For a monoid~$M$, we set $M^\bullet := M  \setminus  \{0\}$, and we let $U(M)$ denote the set of invertible elements of $M$. The monoid $M$ is \emph{reduced} provided that $U(M) = \{0\}$. On the other hand, $M$ is \emph{torsion-free} provided that for all $x,y \in M$ the fact that $nx = ny$ for some $n \in \nn$ implies that $x = y$. For $x,y \in M$, we say that $x$ \emph{divides} $y$ in $M$ and write $x \mid_M y$ if $y = x + x'$ for some $x' \in M$. For $S \subseteq M$, we let $\langle S \rangle$ denote the smallest (under inclusion) submonoid of~$M$ containing~$S$. If $M = \langle S \rangle$, then we call $S$ a \emph{generating set} of $M$. Further basic definitions and concepts on commutative cancellative monoids can be found in~\cite[Chapter~2]{pG01}.

An element $a \in M \setminus U(M)$ is an \emph{atom} if for any $x,y \in M$ with $a = x+y$ either $x \in U(M)$ or $y \in U(M)$. The set of all atoms of $M$ is denoted by $\mathcal{A}(M)$, and $M$ is called \emph{atomic} if each element of $M \setminus U(M)$ can be written as a sum of atoms. A subset $I$ of $M$ is an \emph{ideal} of $M$ if $I + M \subseteq I$. An ideal $I$ is \emph{principal} if $I = x + M$ for some $x \in M$, and $M$ satisfies the \emph{ascending chain condition on principal ideals} (or \emph{ACCP}) provided that every increasing sequence of principal ideals of $M$ eventually stabilizes. It is well known that every monoid satisfying the ACCP must be atomic \cite[Proposition~1.1.4]{GH06}.

The \emph{Grothendieck group} $\gp(M)$ of a monoid $M$ is the abelian group (unique up to isomorphism) satisfying that any abelian group containing a homomorphic image of $M$ will also contain a homomorphic image of $\gp(M)$. The \emph{rank} of the monoid $M$ is defined to be the rank of the group $\gp(M)$, namely, the cardinality of a maximal integrally independent subset of $\gp(M)$. Equivalently, the rank of $M$ is the dimension of the $\qq$-vector space $\qq \otimes_{\zz} \gp(M)$. The monoid $M$ is called a \emph{totally ordered monoid} with respect to a total order~$\preceq$ on the underlying set of $M$ provided that $\preceq$ is compatible with the operation of $M$, i.e., for all $x,y,z \in M$ the inequality $x \preceq y$ implies that $x + z \preceq y + z$. If, in addition, $0 \preceq x$ for all $x \in M$, then we say that $M$ is \emph{positive} with respect to~$\preceq$.

Some of the monoid algebras exhibited in this paper are constructed using submonoids of $(\qq_{\ge 0}, +)$, which are called \emph{Puiseux monoids}. Clearly, each Puiseux monoid is totally ordered with respect to the standard order of $\qq$. Although all Puiseux monoids we shall be using here are atomic, we should notice that there exists a huge variety of non-atomic Puiseux monoids. For instance, for any prime $p$, the Puiseux monoid $M = \langle 1/p^n \mid n \in \nn_0 \rangle$ is not atomic as $\mathcal{A}(M)$ is empty. On the other hand, there are atomic Puiseux monoids that do not satisfy the ACCP, and such monoids are crucial in this paper.

\begin{example} \label{ex:Gram's monoid}
	Let $p_n$ denote the $n^{\text{th}}$ odd prime. The Puiseux monoid
	\[
		G = \bigg\langle \frac{1}{2^n \cdot p_n} \ \bigg{|} \ n \in \nn \bigg\rangle
	\]
	was introduced by A.~Grams in~\cite{aG74} to construct the first example of an atomic integral domain that does not satisfy the ACCP. It is not hard to check that $G$ is an atomic monoid. However, the increasing chain of principal ideals $\{1/2^n + G\}$ does not stabilize and, therefore, $G$ does not satisfy the ACCP. We call $G$ the \emph{Gram's monoid}.
\end{example}

\subsection{Monoid Algebras.} As usual, for a commutative ring $R$ with identity, $R^\times$ denotes the set of units of $R$. An integral domain is \emph{atomic} (resp., satisfies the \emph{ACCP}) if its multiplicative monoid is atomic (resp., satisfies the ACCP).

For a monoid~$M$ and an integral domain $R$, we let $R[x;M]$ denote the ring consisting of all functions $f \colon M \to R$ satisfying that $\{s \in M \mid f(s) \neq 0 \}$ is finite. We shall conveniently write a generic element $f \in R[x;M] \setminus \{0\}$ in one of the following ways:
\[
	f = \sum_{s \in M} f(s)x^s = \sum_{i=1}^n f(s_i)x^{s_i},
\]
where $s_1, \dots, s_n$ are precisely those $s \in M$ with $f(s) \neq 0$. The ring $R[x;M]$ is, indeed, an integral domain~\cite[Theorem~8.1]{rG84}, and the set of units of $R[x;M]$ is $R^\times$ \cite[Theorem~11.1]{rG84}. As a result, $R[x;M]$ is called the \emph{monoid domain} of $M$ \emph{over} $R$ or, simply, a \emph{monoid domain}. When $R$ is a field we call $R[x;M]$ a \emph{monoid algebra}.

Let us assume now that $M$ is a totally ordered monoid with respect to a given order~$\preceq$. In this case, we write the elements $f \in R[x;M] \setminus \{0\}$ in \emph{canonical form}, that is $f = \alpha_1x^{q_1} + \dots + \alpha_k x^{q_k}$ with $\alpha_i \neq 0$ for $i = 1, \dots, k$ and $q_1 \succ \dots \succ q_k$. Clearly, there is only one way to write $f$ in canonical form. The element $q_1 \in M$, denoted by $\deg(f)$, is called the \emph{degree} of $f$. As for the case of rings of polynomials over integral domains, the identity
\[
	\deg(fg) = \deg(f) + \deg(g)
\]
holds for all $f, g \in R[x;M] \setminus \{0\}$. As it is customary for polynomials, we say that $f$ is a \emph{monomial} provided that $k = 1$.

The following two results, which we shall be using later, are well known and have rather straightforward proofs.

\begin{theorem} \label{thm:general known results} \hfill
	\begin{enumerate}
		\item Let $M$ be a monoid satisfying the ACCP. Then for any field $F$ the monoid algebra $F[x;M]$ satisfies the ACCP.
		\vspace{3pt}
		\item Let $R$ be an integral domain. If $R$ satisfies the ACCP, then $R$ is atomic.
	\end{enumerate}
\end{theorem}
\bigskip

\section{Non-Atomic Monoid Domains}
\label{sec:non-atomic monoid domains}

Our goal in this section is to construct an atomic monoid $M$ with the property that for each integral domain $R$ the monoid domain $R[x;M]$ fails to be atomic.

\begin{prop} \label{prop:product of atomic monoids}
	For monoids $M$ and $N$ the following statements hold.
	\begin{enumerate}
		\item If $M$ and $N$ are atomic, then $M \times N$ is atomic.
		\vspace{3pt}
		\item If $M$ and $N$ satisfy the ACCP, then $M \times N$ satisfies the ACCP.
	\end{enumerate}
\end{prop}

\begin{proof}
	To argue~(1), suppose that $M$ and $N$ are atomic. Clearly, $(a,0)$ and $(0,b)$ belong to $\mathcal{A}(M \times N)$ when $a \in \mathcal{A}(M)$ and $b \in \mathcal{A}(N)$. Therefore, for any atomic decompositions $r = \sum_{i=1}^k a_i$ and $s = \sum_{j=1}^\ell b_j$ of $r \in M$ and $s \in N$,
	\[
		(r,s) = \sum_{i=1}^k (a_i,0) + \sum_{j=1}^\ell (0,b_j)
	\]
	is an atomic decomposition of $(r,s)$ in $M \times N$. Hence $M \times N$ is atomic, and~(1) follows. To argue~(2), assume that $M$ and $N$ both satisfy the ACCP. Let $\{ (a_n, b_n) + M \times N\}$ be an increasing chain of principal ideals in $M \times N$. Then $\{a_n + M\}$ and $\{b_n + N\}$ are increasing chains of principal ideals in $M$ and $N$, respectively. As $\{a_n + M\}$ and $\{b_n + N\}$ stabilize, $\{(a_n, b_n) + M \times N\}$ must also stabilize. As a result, $M \times N$ satisfies the ACCP, which completes the proof.
\end{proof}

\begin{theorem} \label{thm:main result 0}
	There exists a reduced torsion-free atomic monoid $M$ with rank $\aleph_0$ such that for each integral domain $R$, the monoid domain $R[x;M]$ is not atomic.
\end{theorem}

\begin{proof}
	Consider the abelian group $G$ freely generated by the set
	\[
		\Omega := \{a,b,c,s_n,t_n \mid n \in \nn_0\}.
	\]
	For each element $g \in G$, one can write
	\[
		g = \sum_{\omega \in \Omega} \mathsf{v}_\omega(g)\omega,
	\]
	where $\mathsf{v}_\omega(g) \in \zz$ and $\mathsf{v}_\omega(g) = 0$ for all but finitely many elements $\omega \in \Omega$. Observe that $\mathsf{v}_\omega(g+h) = \mathsf{v}_\omega(g) + \mathsf{v}_\omega(h)$ for all $g,h \in G$. Clearly, there exists a unique total order $\preceq$ on $G$ satisfying the following conditions:
	\begin{enumerate}
		\item the sequences $\{s_n\}$ and $\{t_n\}$ are both decreasing;
		\vspace{3pt}
		\item $a \succ b \succ c \succ s_n \succ t_0$ for every $n \in \nn_0$;
		\vspace{3pt}
		\item $\preceq$ is lexicographic on $G$ with respect to the order already assigned to the elements of $\Omega$.
	\end{enumerate}
	From now on we treat $G$ as a totally ordered group with respect to the order~$\preceq$. Consider the submonoid $M$ of $G$ generated by the set
	\[
		A := \{ c, s_n, t_n, a - nc - s_n, b - nc - t_n \mid n \in \nn_0 \}.
	\]
	Notice that $a,b \in M$ and $s \succ 0$ for all $s \in A$. As a result, $M$ is a positive monoid with respect to $\preceq$ and, therefore, $M$ is reduced.
	
	To argue that $M$ is atomic it suffices to check that $A$ is a minimal set of generators of $M$. That $c \notin \langle A \setminus \{c\} \rangle$ follows from the fact that $c$ is the only element $s \in A$ with $\mathsf{v}_c(s) \succ 0$. Also, for each $n \in \nn$, $s_n$ is the only element $s \in A$ with $\mathsf{v}_{s_n}(s) \succ 0$ and, therefore, $s_n \notin \langle A \setminus \{s_n\} \rangle$. A similar argument shows that $t_n \notin \langle A \setminus \{t_n\} \rangle$ for any $n \in \nn_0$. In addition, for every $n \in \nn_0$, the element $a - nc - s_n$ is the only $s \in A$ satisfying that $\mathsf{v}_{s_n}(s) \prec 0$, and so $a - nc - s_n \notin \langle A \setminus \{a - nc - s_n\} \rangle$. In a similar way one finds that $b - nc - t_n \notin \langle A \setminus \{b - nc - t_n\} \rangle$ for any $n \in \nn_0$. Hence $A$ is a minimal set of generators of~$M$, which implies that $M$ is an atomic monoid with $\mathcal{A}(M) = A$.
	
	Let us check now that the rank of $M$ is $\aleph_0$. Since $a = nc + s_n + (a - nc - s_n)$, we have that $a \in M$. In a similar way, one can see that $b \in M$. Thus, $\Omega \subseteq M$, which implies that the free commutative monoid $F(\Omega)$ on $\Omega$ is a submonoid of $M$. On the other hand, it is clear that $G$ is the Grothendieck group $F(\Omega)$. Then $G = \gp(F(\Omega)) \subseteq \gp(M) \subseteq G$ and, therefore, $\gp(M) = G$. Hence the rank of $M$ is $\aleph_0$, as desired.
	
	Now let $R$ be an integral domain. We proceed to show that $x^a + x^b \in R[x;M]$ cannot be expressed as a product of irreducibles. Suppose, by way of contradiction, that $x^a + x^b = f g$ for some $f,g \in R[x;M]$ such that $f$ is an irreducible in $R[x;M]$ with $\mathsf{v}_a(\deg(f)) \succ 0$ and $g \notin R[x;M]^\times = R^\times$. Write
	\[
		f = \sum_{j=1}^n r_j x^{v_j} \quad \text{ and } \quad g = \sum_{j=1}^m h_j x^{w_j}
	\]
	in canonical forms. As $v_1 + w_1 = a \in A$, one finds that $\mathsf{v}_a(v_1) = 1$ and $\mathsf{v}_a(w_1) = 0$. In addition, for $j \in \ldb 1,m \rdb$ we have that $w_j \preceq w_1$ and so $\mathsf{v}_a(w_j) = 0$. Now $v_n + w_m = b$ implies that $\mathsf{v}_a(v_n) = 0$ and, therefore, there exists a smallest index $i$ with $\mathsf{v}_a(v_i) = 0$. As $\mathsf{v}_a(v_1) = 1$ we have that $i > 1$.
	\vspace{5pt}
	
	\noindent \textit{Claim 1.} $v_{i'} + w_j \neq v_i + w_1$ when $(i',j) \neq (i,1)$.
	\vspace{2pt}
	
	\noindent \textit{Proof of Claim 1.} First, note that if $i' > i$, then $v_{i'} \prec v_i$ and $w_j \preceq w_1$; this implies that $v_{i'} + w_j \prec v_i + w_1$. On the other hand, if $i' = i$, then $j > 1$ and so $w_j \prec w_1$; this also implies that $v_{i'} + w_j \prec v_i + w_1$. Lastly, if $i' < i$, then $\mathsf{v}_a(v_{i'}) > 0$ by the minimality of the index $i$. This implies that $\mathsf{v}_a(v_{i'} + w_j) > 0 = \mathsf{v}_a(v_i + w_1)$, and our claim follows.
	\vspace{3pt}
	
	By Claim 1, we have that $x^{v_i + w_1}$ occurs in $x^a + x^b$. Because $\mathsf{v}_a(v_i + w_1) = 0$, the equality $x^{v_i + w_1} = x^b$ must hold and, as a consequence, $v_i + w_1 = v_n + w_m$. Thus, $m = 1$ and so it immediately follows that $g = x^s$ for some $s \in M$.
	\vspace{5pt}
	
	\noindent \textit{Claim 2.} $s \in \langle c \rangle$.
	\vspace{2pt}
	
	\noindent \textit{Proof of Claim 2.} It is clear that $s$ is a common divisor of $a$ and $b$ in $M$. Take $a', b' \in M$ such that $a = a' + s$ and $b = b' + s$. From $a = a' + s$, one obtains that $0 \le \max \{ \mathsf{v}_b(a'), \mathsf{v}_b(s) \} \le \mathsf{v}_b(a) = 0$, and so $\mathsf{v}_b(a') = \mathsf{v}_b(s) = 0$. This, in turn, implies that $\mathsf{v}_{t_n}(a') \ge 0$ and $\mathsf{v}_{t_n}(s) \ge 0$ for every $n \in \nn_0$. Thus, $\mathsf{v}_{t_n}(s) = 0$ for every $n \in \nn_0$. In a similar manner one can verify that $\mathsf{v}_a(s) = \mathsf{v}_{s_n}(s) = 0$ for every $n \in \nn_0$. Thus, $s \in \langle c \rangle$, as claimed.
	
	Finally, suppose that $s = kc$ for some $k \in \nn$. Then $f = x^{a - kc} + x^{b - kc}$. For $c' := (a - (k+2)c - s_{k+2}) + s_{k+2} \in M$, one sees that $2c + c' = a - kc$. Thus, $2c \mid_M a - kc$. This implies that $x^{2c} \mid_{R[x;M]} x^{a - kc}$. Similarly, we have that $x^{2c} \mid_{R[x;M]} x^{b - kc}$. As a result, the non-unit non-irreducible element $x^{2c}$ divides $f$ in $R[x;M]$, which contradicts the fact that $f$ is irreducible in $R[x;M]$.
	Then $x^a + x^b$ cannot be expressed as a product of irreducibles in $R[x;M]$, which implies that $R[x;M]$ is not atomic.
\end{proof}

For any integral domain $R$ and monoids $M$ and $N$, there is a canonical ring isomorphism $F[x; M \times N] \cong F[y;M][z;N]$ induced by the assignment $x^{(a,b)} \mapsto y^a z^b$. To avoid having ordered pairs as exponents, we identify $R[x;M \times N]$ with $F[y;M][z;N]$ and write the elements of $F[x, M \times N]$  as polynomial expressions in two variables.

\begin{cor} \label{cor:main result 0}
	For any infinite cardinal $\kappa$, there exists a reduced torsion-free atomic monoid $M$ with rank $\kappa$ such that for any integral domain $R$ the monoid domain $R[x;M]$ is not atomic.
\end{cor}

\begin{proof}
	By Theorem~\ref{thm:main result 0} there exists a reduced torsion-free atomic monoid $M'$ with rank~$\aleph_0$ such that for any integral domain $R$, the monoid domain $R[x;M']$ is not atomic. Let $M_\kappa$ be the free commutative monoid of rank $\kappa$, and set $M := M' \times M_\kappa$. Clearly, $M$ is reduced and torsion-free. It follows from Proposition~\ref{prop:product of atomic monoids}(1) that $M$ is atomic, and the equality $\gp(M) = \gp(M') \times \gp(M_\kappa)$ ensures that the rank of $M$ is~$\kappa$. Now let $R$ be an integral domain. Then $R[x;M] \cong R[y;M'][z;M_\kappa]$. Since $M_\kappa$ is a torsion-free monoid and the monoid domain $R[y;M']$ is not atomic, \cite[Proposition~1.4]{hK01} guarantees that $R[x;M]$ is not atomic.
\end{proof}

In the next two sections, we shall construct reduced torsion-free atomic monoids with monoid algebras that are not atomic over finite-characteristic fields.
\bigskip

\section{Non-atomic Monoid Algebras of Finite Characteristic}
\label{sec:non-atomic algebras of finite characteristic}

In this section, we find, for any given field $F$ of finite characteristic, a rank-$2$ totally ordered atomic monoid $M$ such that $F[x;M]$ is not atomic.

Let $r$ and $m$ be integers with $m > 0$ and $\gcd(r,m) = 1$. Recall that the order of $r$ modulo $m$ is the smallest $n \in \mathbb{N}$ for which $r^n \equiv 1 \pmod{m}$, and that $r$ is a primitive root modulo $m$ if its order modulo $m$ equals $\phi(m)$, where $\phi$ is the Euler's totient function. It is well known that for any odd prime $p$ and positive integer $k$, there exists a primitive root modulo $p^k$.

\begin{lemma} \label{lem:irreducible polynomials in two variables}
	Let $F$ be a field of finite characteristic $p$ and $n \in \nn$ be such that $p \nmid n$. Then the polynomial $x^n+y^n+x^ny^n$ is irreducible in $F[x,y]$.
\end{lemma}

\begin{proof}
	Set $f(x,y) = y^n(1+x^n)+x^n$. Since $1 + x^n$ and $x^n$ are relatively primes in $F[x]$,  the polynomial $f(x,y)$ is primitive as a polynomial on $y$ over $F[x]$. By Gauss's Lemma, arguing that $f(x,y)$ is irreducible in $F[x][y]$ amounts to proving that it is irreducible in $F(x)[y]$, where $F(x)$ is the field of fractions of $F[x]$. We can write now
	\[
		f(x,y) = (1+x^n)y^n + x^n = (1+x^n) \bigg( y^n + \frac{x^n}{1+x^n} \bigg).
	\]
	Set $a_x = \frac{x^n}{1+x^n}$. Then $f(x,y)$ is irreducible in $F[x,y]$ if and only if $y^n + a_x$ is irreducible in $F(x)[y]$. Using~\cite[Theorem~8.1.6]{gK89}, one can guarantee the irreducibility of $y^n + a_x$ by verifying that $a_x \notin 4 F(x)^4$ when $4$ divides $n$ and that $-a_x \notin F(x)^q$ for any prime $q$ dividing $n$. To prove that these two conditions hold suppose, by way of contradiction, that $a_x \in c F(x)^q$, where $c \in \{-1,4\}$ and $q$ is either $4$ or a prime dividing $n$. Take $h_1(x), h_2(x) \in F[x] \setminus \{0\}$ such that $h_1(x)$ and $h_2(x)$ are relatively prime in $F[x]$ and $a_x = c \big( h_1(x)/h_2(x) \big)^q$. Then we have that
	\begin{equation} \label{eq:equality of rational fractions}
		x^n h_2(x)^q = c(1 + x^n)h_1(x)^q
	\end{equation}
	From~\eqref{eq:equality of rational fractions}, one can deduce that $h_2(x)^q$ and $1+x^n$ are associates in $F[x]$, and so there exists $\alpha \in F^\times$ such that $h_2(x)^q = \alpha(1+x^n)$. Taking derivatives in both sides of $h_2(x)^q = \alpha(1+x^n)$ and using that $p \nmid n$, we obtain that $h_2(x) = x^m$ for some $m \in \nn$, yielding that $c(1+x^n)h_1(x)^q = x^{n+mq}$. However, this contradicts that $1+x^n$ does not divide $x^{n+mq}$ in $F[x]$. Hence $f(x,y)$ is irreducible in $F[x,y]$.
\end{proof}
\medskip

Motivated by the Gram's monoid, in the next example we exhibit a family of Puiseux monoids indexed by prime numbers whose members are atomic but do not satisfy the ACCP.

\begin{example} \label{ex:like-Gram monoids}
	Let $\{p_n\}$ be a sequence consisting of all prime numbers ordered increasingly. For each prime $p$ consider the Puiseux monoid
	\[
		M_p := \bigg\langle \frac{1}{p^n p_n} \ \bigg{|} \ p_n \neq p \bigg\rangle.
	\]
	A very elementary argument of divisibility can be used to check that $M_p$ is atomic for each prime $p$. On the other hand, $M_p$ contains the strictly increasing sequence of principal ideals $\{1/p^n + M_p \}$ and, therefore, $M_p$ does not satisfy the ACCP. Notice that $M_2$ is precisely the Gram's monoid.
\end{example}


\begin{theorem} \label{thm:main result 1}
	For each field $F$ of finite characteristic $p$, there exists a torsion-free rank-$2$ atomic monoid $M$ such that the monoid algebra $F[x;M]$ is not atomic.
\end{theorem}

\begin{proof}
	Take $M := M_p \times M_p$, where $M_p$ is the atomic monoid parametrized by $p$ exhibited in Example~\ref{ex:like-Gram monoids}. It is clear that $M$ is torsion-free and has rank $2$. In addition, $M$ is atomic by Proposition~\ref{prop:product of atomic monoids}(1). Now we claim that each non-unit factor of $f := X + Y + XY$ in $F[x;M]$ has the form
	\[
		\big( X^{\frac{1}{p^k}} + Y^{\frac{1}{p^k}} + X^{\frac{1}{p^k}}Y^{\frac{1}{p^k}} \big)^t
	\]
	for some $k \in \nn_0$ and $t \in \nn$. To prove our claim, let $g \in F[x;M]$ be a non-unit factor of $f$, and take $h \in F[x;M]$ such that $f = g \, h$. Then there exist $k \in \nn_0$ and $a \in \nn$ with $p \nmid a$ such that $g(X^{ap^k}, Y^{ap^k})$ and $h(X^{ap^k}, Y^{ap^k})$ are both in the polynomial ring $F[X,Y]$. After changing variables, we obtain
	\[
		g(X^{ap^k},Y^{ap^k})h(X^{ap^k},Y^{ap^k}) = X^{ap^k} + Y^{ap^k} + X^{ap^k}Y^{ap^k} = (X^a + Y^a + X^aY^a)^{p^k}.
	\]
	By Lemma~\ref{lem:irreducible polynomials in two variables}, the polynomial $X^a + Y^a + X^aY^a$ is irreducible in the polynomial ring $F[X,Y]$. Since $F[X,Y]$ is a UFD, there exists $t \in \nn$ such that
	\[
		g\big( X^{ap^k}, Y^{ap^k} \big) = \big(X^a + Y^a + X^aY^a \big)^t.
	\]
	Going back to the original variables, we obtain $g(X,Y) = \big( X^{\frac{1}{p^k}}+ Y^{\frac{1}{p^k}} + X^{\frac{1}{p^k}} Y^{\frac{1}{p^k}} \big)^t$, which establishes our initial claim.

	Let us proceed to verify that $F[x;M]$ is not atomic. Notice that $f$ is not irreducible as $f = \big( X^\frac1p + Y^\frac1p + X^\frac1p Y^\frac1p \big)^p$. By the argument given in the previous paragraph, any factor $g$ of $f$ in a potential decomposition into irreducibles of $F[x;M]$ must be of the form $\big( X^{\frac{1}{p^k}} + Y^{\frac{1}{p^k}} + X^{\frac{1}{p^k}}Y^{\frac{1}{p^k}} \big)^t$ and, therefore,
	\begin{align} \label{eq:main theorem I}
		g = \big( X^{\frac{1}{p^{k+1}}} + Y^{\frac{1}{p^{k+1}}} + X^{\frac{1}{p^{k+1}}}Y^{\frac{1}{p^{k+1}}} \big)^{pt}.
	\end{align}
	Since $X^{\frac{1}{p^{k+1}}} + Y^{\frac{1}{p^{k+1}}} + X^{\frac{1}{p^{k+1}}}Y^{\frac{1}{p^{k+1}}} \in F[x;M]$, the equality~(\ref{eq:main theorem I}) would contradict that $g$ is an irreducible element of $F[x;M]$. Thus, the algebra $F[x;M]$ is not atomic.
\end{proof}


\begin{cor} \label{cor:linearly ordered rank-2 examples}
	For each field $F$ of finite characteristic $p$ and each $r \in \nn_{\ge 2}$, there exists a totally ordered atomic monoid $M$ with rank $r$ such that $F[x;M]$ is not atomic.
\end{cor}

\begin{proof}
	By Theorem~\ref{thm:main result 1}, there exists a torsion-free rank-$2$ atomic monoid $M'$ such that the monoid algebra $F[y;M']$ is not atomic. Now take $M := M' \times \nn_0^{r-2}$. Clearly, $M$ has rank $r$. On the other hand, Proposition~\ref{prop:product of atomic monoids}(1) ensures that $M$ is atomic. Also, $M$ is totally orderable because it is torsion-free and cancellative~\cite[Corollary~3.4]{rG84}. Finally, since $F[x;M] \cong F[y;M'][z;\nn_0^{r-2}]$, it follows from~\cite[Proposition~1.4]{hK01} and the fact that $F[y;M']$ is not atomic that $F[x;M]$ is not atomic.
\end{proof}
\medskip

%

\section{Non-atomic Monoid Algebras with Rational Exponents}
\label{sec:PM}

The purpose of this section is to find an atomic Puiseux monoid $M$ such that the algebra $\zz_2[x;M]$ is not atomic. Since every Puiseux monoid is totally ordered and has rank $1$, this result will complement Corollary~\ref{cor:linearly ordered rank-2 examples}.

For $q \in \qq_{> 0}$, let $a,b \in \nn$ be the unique positive integers such that $q = a/b$ and $\gcd(a,b)=1$. Then we denote $b$ by $\mathsf{d}(q)$. In addition, if $S \subseteq \qq_{> 0}$, then we denote the set $\{\mathsf{d}(s) \mid s \in S\}$ by $\mathsf{d}(S)$.

\begin{prop} \label{prop:an atomic Puisuex monoid}
	There exists an atomic Puiseux monoid $M$ satisfying the following two conditions:
	\begin{enumerate}
		\item $M \subseteq \zz\big[ \frac{1}{2}, \frac{1}{3} \big]$;
		\vspace{3pt}
		\item $\frac{1}{2^n} \in M$ for each $n \in \mathbb{N}_0$.
	\end{enumerate}
\end{prop}

\begin{proof}
	Let $\{\ell_n\}$ be a strictly increasing sequence of positive integers satisfying that
	\[
		3^{\ell_n - \ell_{n-1}} > 2^{n+1}
	\]
	for every $n \in \nn$. Now set $A = \{a_n, b_n \mid n \in \nn\}$, where $a_n := \frac{2^n3^{\ell_n} - 1}{2^{2n} 3^{\ell_n}}$ and $b_n := \frac{2^n3^{\ell_n} + 1}{2^{2n} 3^{\ell_n}}$. It is clear that $1 > b_n > a_n$ for every $n \in \nn$. In addition,
	\[
		a_n = \frac{1}{2^n} - \frac{1}{2^{2n}3^{\ell_n}} = \frac{1}{2^{n+1}} + \bigg( \frac{1}{2^{n+1}} - \frac{1}{2^{2n}3^{\ell_n}} \bigg) > \frac{1}{2^{n+1}} + \frac{1}{2^{2n+2}3^{\ell_{n+1}}} = b_{n+1}
	\]
	for every $n \in \nn$. Therefore the sequence $b_1, a_1, b_2, a_2, \dots$ is strictly decreasing and bounded from above by $1$. Consider now the Puiseux monoid $M = \langle A \rangle$. Clearly, $M$ satisfies condition~(1). On the other hand, $\frac{1}{2^n} = a_{n+1} + b_{n+1} \in M$ for every $n \in \nn_0$. Thus, $M$ also satisfies condition~(2).
	
	All we need to prove is that $M$ is atomic. It suffices to verify that $A$ is a minimal generating set of $M$~\cite[Proposition~1.1.7]{GH06}. Suppose, by way of contradiction, that this is not the case. Then there exists $n \in \nn$ such that $M = \langle A \setminus \{a_n\} \rangle$ or $M = \langle A \setminus \{b_n\} \rangle$.
	\vspace{3pt}
	
	\noindent CASE~1: $M = \langle A \setminus \{a_n\} \rangle$. In this case,
	\begin{equation} \label{eq:testing atoms I}
		a_n = \sum_{i=1}^N \alpha_i a_i + \sum_{i=1}^N \beta_i b_i
	\end{equation}
	for some $N \in \nn_{\ge n}$ and nonnegative integer coefficients $\alpha_i$'s and $\beta_i$'s ($i \in \ldb 1,N \rdb$) such that $\alpha_n = 0$ and either $\alpha_N > 0$ or $\beta_N > 0$. Since the sequence $b_1, a_1, b_2, a_2, \dots$ is strictly decreasing, $\alpha_i = \beta_i = 0$ for $i \in \ldb 1, n \rdb$. Notice that $\alpha_i = \beta_i$ cannot hold for all $i \in \ldb n+1, N \rdb$; otherwise,
	\[
		a_n = \sum_{i=n+1}^N \alpha_i a_i + \sum_{i=n+1}^N \alpha_i b_i= \sum_{i = n+1}^N \alpha_i \frac{1}{2^{i-1}},
	\]
	which is impossible because $3 \mid \mathsf{d}(a_n)$. Set $m = \max \big\{ i 
	\in \ldb n+1, N \rdb \mid \alpha_i \neq \beta_i \big\}$. First assume that $\alpha_m > \beta_m$. Then we can rewrite~(\ref{eq:testing atoms I}) as follows:
	\begin{equation} \label{eq:testing atoms III}
		a_n = (\alpha_m - \beta_m) \frac{2^m3^{\ell_m} - 1}{2^{2m} 3^{\ell_m}} + \sum_{i=m}^N \beta_i \frac{1}{2^{i-1}} + \! \! \sum_{i=n+1}^{m-1} \frac{\alpha_i(2^i3^{\ell_i} - 1) + \beta_i(2^i3^{\ell_i} + 1)}{2^{2i} 3^{\ell_i}}.
	\end{equation}
	After multiplying both sides of the equality~(\ref{eq:testing atoms III}) by $2^{2N}3^{\ell_m}$, one can easily see that each summand involved in such an equality except perhaps $2^{2N - 2m}(\alpha_m - \beta_m)(2^m 3^{\ell_m} - 1)$ is divisible by $3^{\ell_m -\ell_{m-1}}$. Therefore $3^{\ell_m - \ell_{m-1}}$ must also divide $\alpha_m - \beta_m$. Now since $a_m > b_{m+1} > \frac{1}{2^{m+1}}$, we find that
	\begin{equation} \label{eq:testing atoms IV}
		a_n \ge \alpha_m a_m \ge (\alpha_m - \beta_m) b_{m+1} \ge 3^{\ell_m - \ell_{m-1}} b_{m+1} > \frac{3^{\ell_m - \ell_{m-1}}}{2^{m+1}} > 1,
	\end{equation}
	which is a contradiction. In a similar way we arrive at a contradiction if we assume that $\beta_m > \alpha_m$.
	
	\vspace{3pt}
	\noindent CASE~2: $M = \langle A \setminus \{b_n\} \rangle$. In this case, it is not hard to see that
	\begin{equation} \label{eq:testing atoms II}
		b_n - \alpha_n a_n = \sum_{i=n+1}^N \alpha_i a_i + \sum_{i=n+1}^N \beta_i b_i
	\end{equation}
	for some nonnegative coefficients $\alpha_i$'s ($i \in \ldb n, N \rdb$) and $\beta_j$'s ($j \in \ldb n+1,N \rdb$) such that either $\alpha_N > 0$ or $\beta_N > 0$. Observe that $\alpha_n \in \{0,1\}$ as it is obvious that $2 a_n > b_n$. As before, there exists $m \in \ldb n+1, N \rdb$ such that $\alpha_m \neq \beta_m$, and we can assume that such $m$ is as large as possible. If $\alpha_m > \beta_m$, then
	\[
		b_n - \alpha_n a_n = (\alpha_m - \beta_m) \frac{2^m3^{\ell_m} - 1}{2^{2m} 3^{\ell_m}} + \sum_{i=m}^N \beta_i \frac{1}{2^{i-1}} + \! \! \sum_{i=n+1}^{m-1} \frac{\alpha_i(2^i3^{\ell_i} - 1) + \beta_i(2^i3^{\ell_i} + 1)}{2^{2i} 3^{\ell_i}}.
	\]
	Since $\mathsf{d}(b_n - \alpha_n a_n) \in \{2^{n-1} 3^{\ell_n}, 2^{2n} 3^{\ell_n}\}$, after multiplying the previous equation by $2^{2N}3^{\ell_m}$ we can see that $3^{\ell_m - \ell_{m-1}}$ divides $\alpha_m - \beta_m$. Now, an argument similar to that one given in CASE~1 can be used to obtain that $b_n > 1$, which is a contradiction. We can proceed in a similar manner to obtain a contradiction if we assume that $\alpha_m < \beta_m$. Hence we have proved that $A$ is a minimal generating set of $M$, which means that $M$ is atomic with $\mathcal{A}(M) = A$.
\end{proof}
\medskip

The following result will be used in the proof of Lemma~\ref{lem:irreducible polynomials in Z2[X]}.

\begin{lemma} \label{lem:primitive roots} \cite[page~179]{KR98}
	If $p$ is an odd prime and $r$ is a primitive root modulo $p^2$, then $r$ is a primitive root modulo $p^n$ for every $n \ge 2$.
\end{lemma}

For a given field $F$, we let $Q_n(x)$ denote the $n^{\text{th}}$ cyclotomic polynomial over $F$. The next lemma is proposed as an exercise in~\cite[Chapter~3]{LN86}. For the convenience of the reader, we provide a proof here.

\begin{lemma} \label{lem:irreducible polynomials in Z2[X]}
	For each $n \in \mathbb{N}$, the polynomial $x^{2 \cdot 3^n} + x^{3^n} + 1$ is irreducible in $\mathbb{Z}_2[x]$.
\end{lemma}

\begin{proof}
	One can verify that the order of $2$ modulo $9$ is $\phi(9) = 6$. Then $2$ is a primitive root modulo $3^2$. It follows now by Lemma~\ref{lem:primitive roots} that $2$ is also a primitive root modulo $3^k$ for every $k \ge 2$. Fix $n \in \nn$ and set $f_n(x) := x^{2 \cdot 3^n} + x^{3^n} + 1$. From
	\[
		x^{3^{n+1}} - 1 = \prod_{i=0}^{n+1} Q_{3^i}(x)
	\]
	one obtains that
	\[
		Q_{3^{n+1}}(x) = \frac{x^{3^{n+1}} - 1}{\prod_{i=0}^n Q_{3^i}(x)} = \frac{x^{3^{n+1}} - 1}{x^{3^n} - 1} = f_n(x).
	\]
	Therefore $f_n(x)$ is the ${3^{n+1}}$-th cyclotomic polynomial over $\zz_2$ (see \cite[Example~2.46]{LN86}). Since $2$ is a primitive root module $3^k$ for any $k \ge 2$, the least positive integer $d$ satisfying that $2^d  \equiv 1 \pmod{3^{n+1}}$ is $\phi(3^{n+1}) = 2 \cdot 3^n$. Hence~\cite[Theorem~2.47(ii)]{LN86} guarantees that the polynomial $f_n(x)$ is irreducible.
\end{proof}

%

\begin{theorem} \label{thm:main result 2}
	There exists an atomic Puiseux monoid $M$ such that $\mathbb{Z}_2[x;M]$ is not atomic.
\end{theorem}

\begin{proof}
	Let $M$ be an atomic Puiseux monoid satisfying conditions~(1) and (2) of Proposition~\ref{prop:an atomic Puisuex monoid}. First, we will argue that each factor of the element $x^2 + x + 1$ in $\mathbb{Z}_2[x;M]$ has the form $\big( x^{2 \frac{1}{2^k}} + x^{\frac{1}{2^k}} + 1 \big)^t$ for some $k \in \mathbb{N}_0$ and $t \in \mathbb{N}$. First, note that because $M$ contains $\langle 1/2^k \mid k \in \mathbb{N}_0 \rangle$, it follows that $x^{2 \frac{1}{2^k}} + x^{\frac{1}{2^k}} + 1 \in \mathbb{Z}_2[x;M]$ for all $k \in \nn_0$. Now suppose that $f(x)$ is a factor of $x^2 + x + 1$ in $\mathbb{Z}_2[x;M]$, and take $g(x) \in \mathbb{Z}_2[x;M]$ such that $x^2 + x + 1 = f(x) g(x)$. Then there exists $k \in \mathbb{N}_0$ such that
	\[
		f\big( x^{6^k}\big) g\big( x^{6^k} \big) = \big( x^{6^k} \big)^2 + x^{6^k} + 1 = \big( x^{2 \cdot 3^k} + x^{3^k} + 1 \big)^{2^k}
	\]
	in the polynomial ring $\mathbb{Z}_2[x]$. By Lemma~\ref{lem:irreducible polynomials in Z2[X]}, the polynomial $x^{2 \cdot 3^k} + x^{3^k} + 1$ is irreducible in $\mathbb{Z}_2[x]$. Since $\mathbb{Z}_2[x]$ is a UFD, there exists $t \in \mathbb{N}$ such that
	\begin{align} \label{eq:factor description}
		f \big( x^{6^k}\big) = \big( x^{2 \cdot 3^k} + x^{3^k} + 1 \big)^t = \big( \big( x^{6^k} \big)^{2 \frac{1}{2^k}} + \big( x^{6^k} \big)^{\frac{1}{2^k}} + 1 \big)^t.
	\end{align}
	After changing variables in~(\ref{eq:factor description}), one obtains that $f(x) = \big( x^{2 \frac{1}{2^k}} + x^{\frac{1}{2^k}} + 1 )^t$. Thus, each factor of $x^2 + x + 1$ in $\mathbb{Z}_2[x;M]$ has the desired form.
	
	Now suppose, by way of contradiction, that the monoid domain $\mathbb{Z}_2[x;M]$ is atomic. Then $x^2 + x + 1 = \prod_{i=1}^n f_i(x)$ for some $n \in \mathbb{N}$ and irreducible elements $f_1(x), \dots, f_n(x)$ in $\mathbb{Z}_2[x;M]$. Since $f_1(x)$ is a factor of $x^2 + x + 1$, there exist $k \in \mathbb{N}_0$ and $t \in \mathbb{N}$ such that $f_1(x) = \big( x^{2 \frac{1}{2^k}} + x^{\frac{1}{2^k}} + 1 )^t$. As $f_1(x)$ is irreducible, $t=1$. Now the equality $f_1(x) = \big( x^{2 \frac{1}{2^{k+1}}} + x^{\frac{1}{2^{k+1}}} + 1 \big)^2$ contradicts the fact that $f_1(x)$ is irreducible in $\mathbb{Z}_2[x;M]$. Hence $\mathbb{Z}_2[x;M]$ is not atomic.
\end{proof}
\medskip

\section*{Acknowledgements}
	While working on this paper, the second author was partially supported by the UC Year Dissertation Fellowship. The authors would like to thank an anonymous referee for her/his careful review and very helpful suggestions.

\end{document}